\documentclass[11pt,a4paper,reqno]{amsart}
\usepackage[english]{babel}
\usepackage[utf8]{inputenc} 
\usepackage{fancyhdr}
\usepackage{indentfirst}
\usepackage{color}     
\usepackage{graphicx}   
\usepackage{newlfont}

\usepackage{amssymb}
\usepackage{amsmath}
\usepackage{latexsym}
\usepackage{amsthm}
\usepackage{mathrsfs}
\usepackage{verbatim}
\usepackage[all]{xy}
\usepackage{lineno}
\usepackage{hyperref}
\usepackage{amsmath,amsthm,amsfonts,amscd,amssymb,mathrsfs,bbm,enumerate,url}

\theoremstyle{plain} 
\newtheorem*{theo*}{Theorem}
\newtheorem{theo}{Theorem}[section] 
\newtheorem{prop}[theo]{Proposition}

\newtheorem{lem}[theo]{Lemma}
\newtheorem*{con*}{Conjecture}

\theoremstyle{definition}
\newtheorem{defin}[theo]{Definition}

\theoremstyle{remark}
\newtheorem{rem}[theo]{Remark}

\newcommand{\CC}{{\mathbb C}}

\newcommand{\ZZ}{{\mathbb Z}}

\newcommand{\Zgeq}{\ZZ_{\geq 0}}

\begin{document}

\author{Sebastiano Carpi}
\address{Dipartimento di Matematica, Universit\`a di Roma ``Tor Vergata'', Via della Ricerca Scientifica, 1, I--00133 Roma, Italy
\\
E-mail: {\tt carpi@mat.uniroma2.it}}
\author{Luca Tomassini}
\address{Liceo Scientifico Statale ``J.F. Kennedy", Via Nicola Fabrizi, 7,  I--00153 Roma, Italy \\
E-mail: {\tt luca.tomassini@liceokennedy.it}}

\title{Energy bounds for vertex operator algebra extensions}

\date{26 May 2023}

\begin{abstract}
Let $V$ be a simple unitary vertex operator algebra and $U$ be a (polynomially) energy-bounded unitary subalgebra containing the conformal vector of $V$. 
We give two sufficient conditions implying that $V$ is energy-bounded.  The first condition is that  $U$ is a compact orbifold $V^G$ for some compact group $G$ of unitary automorphisms of $V$. The second condition is that $V$ is exponentially energy-bounded and it is a finite direct sum of simple $U$-modules. As consequence of the second condition, we prove that if $U$ is a regular energy-bounded unitary subalgebra of a simple unitary vertex operator $V$,  then $V$ is energy-bounded. In particular, every simple unitary extension (with the same conformal vector) of a simple unitary affine vertex operator algebra associated with a semisimple Lie algebra is  energy-bounded.   
\end{abstract}

\maketitle

\begin{section}{Introduction}
Since more than fifty years energy bounds play an important role in the mathematics of quantum field theory, see, e.g.,  \cite{Buch90,DF77,DSW,GJ87,Nel72}. Roughly speaking, a quantum field theory is (polynomially) energy-bounded if the corresponding smeared field operators can be bounded by some power of the Hamiltonian operator. The typical applications of energy bounds in this context are the self-adjointness of smeared field operators and the connection between the Wightman and the Haag-Kastler approaches to quantum field theory. 

In more recent years energy bounds appeared naturally in the study of chiral (two-dimensional) conformal field theories 
(see, e.g., \cite{BS90,Tol97,Was98}) and in the representation theory of the infinite-dimensional Lie groups describing the symmetries of these theories (see, e.g.,\cite{CDVIT22,GW84,GW85,Tol99}). In these cases the role of the Hamiltonian 
operator is typically played by the {\it conformal Hamiltonian} $L_0$. Applications of the energy bounds for chiral conformal field theories include quantum energy inequalities \cite{FH05}, uniqueness of conformal symmetry \cite{CW05}, construction of spectral triples \cite{CHKL10,CHKLX,CHL}, approximation of conformal field theories \cite{KS17,ZW18}, modular flow, relative entropy and quantum null energy conditions \cite{Hol20,Hol21,Pan20,Pan21}, functional analytic properties of Segal's conformal field theories \cite{Ten19a,Ten19b}, and construction of conformal nets on two-dimensional space-times \cite{AGT23,JT23}. Moreover, the energy bonds are the starting point to prove the more sophisticated {\it local energy bounds} \cite{CTW22,CW}.

A systematic use of  the energy bounds for unitary vertex operator algebras was started in \cite{CKLW18} in order to investigate their connections with conformal nets from a general point of view. Subsequently, the study of the relations between vertex operator algebras and conformal nets has been considerably developed in order to cover the representation theory aspects, see, e.g.,  \cite{CWX,Gui19a,Gui20,Gui21a,Ten19a,Ten19b,Ten19c}. The energy bounds also play an essential role for these representation theory aspects.

Although energy bounds for vertex operator algebras appear difficult to prove in general, many important examples of 
unitary vertex operator algebras have been shown to be energy-bounded. Among them we mention: unitary Virasoro vertex operator algebras, unitary affine vertex operator algebras, unitary lattice vertex operator algebras, the moonshine vertex operator algebra, the even shorter moonshine vertex operator algebras and their unitary subalgebras including orbifolds and cosets  \cite{CKLW18,Gui20,Gui21a}. It has been conjectured that every unitary simple vertex operator algebra is energy-bounded \cite[Conjecture 8.18]{CKLW18} and this conjecture appears to be presently open.
A weak version of the energy bounds (the {\it uniformly bounded order property}) has been recently proved to hold for all unitary vertex operator algebras (and many other unitary vertex algebras) in \cite{RTT22}. Although the uniformly bounded order property can be used to replace the existence of energy bounds for various purposes,  ordinary energy bounds still appear to be important, if not essential, in various situations and it is presently not clear how to derive the stronger version of the bounds from the weaker one. 

The easiest examples of energy-bounded unitary vertex operator algebras, such as the affine and Virasoro unitary vertex operator algebras follow in a rather direct way  from the Lie algebra relations  of the generating fields, cf.\  Section \ref{section_bounds}.  Many other examples come by considering unitary subalgebras. What about vertex operator algebra extensions such as simple current or exotic (cf.\ \cite{Gan23}) extensions? Some of them can be can be shown to be energy-bounded 
by proving energy bounds for sufficiently many intertwining   operators of the smaller vertex operator algebra. This method has been developed by Gui and used to prove, among other things, that the extensions of unitary affine vertex operator algebras of $ADE$ Lie type and the lattice vertex operator algebras (viewed as extensions of Heisenberg vertex operator algebras) are energy-bounded, see  \cite{Gui20,Gui21a}. Unfortunately, the proof of the energy bounds for the intertwining operators can be in general very hard. For example, it is presently unknown if there are enough intertwining operators satisfying energy bounds in order to prove that the extensions of the affine vertex operator algebras are energy-bounded for all Lie types besides ADE.

The main results of this paper are certain conditions implying that a simple unitary vertex operator algebra $V$ is energy-bounded if it contains a 
``sufficiently large'' energy-bounded unitary subalgebra $U$. For example we prove that if a unitary compact orbifold $V^G$ is energy-bounded, then $V$ is energy-bounded  (Theorem \ref{VGtheorem}). Moreover, we prove that if a regular unitary subalgebra $U$ of $V$ with the same conformal vector of $V$ is generated by vectors with conformal energy 1, Virasoro vectors and $\mathcal{W}_3$ vectors, then $V$ is energy-bounded 
(Theorem \ref{regularextension_theorem2}). This gives many new examples of energy-bounded unitary vertex operator algebras. In particular, 
every simple vertex operator algebra extension (with the same conformal vector) of a rational, simple unitary affine vertex operator algebra of {\it any} Lie type is unitary and energy-bounded (Theorem \ref{theoaffineextension} ).

\end{section}

\begin{section}{Unitary vertex operator algebras}
\label{section_unitary}

In this section, we briefly discuss some preliminaries about unitary vertex operator algebras and their unitary subalgebras. 
We mainly follow the terminology and notations  in \cite{CKLW18}, see also \cite{DL14}. For the general theory of vertex (operator) algebras needed in this paper, we refer the reader to \cite{FHL,Kac98,LL04}.   In this paper, every vertex operator algebra and in fact every vector space will be over the field $\CC$ of complex numbers.

Let now $V$ be a vertex operator algebra. We denote by $\Omega$ and $\nu$ the vacuum vector and the conformal vector of $V$ 
respectively. Moreover, we denote by $1_V$ the identity in $\operatorname{End(V)}$ (more generally, for any vector space $X$ we denote by $1_X$ the identity in $\operatorname{End(X)}$ ). For every $a \in V$ the vertex operator
\begin{equation}
Y(a,z) = \sum_{n \in \ZZ}a_{(n)}z^{-n-1}
\end{equation} 
is a (quantum) field on $V$, i.e.\ a formal Laurent series with coefficients in $\operatorname{End}(V)$ such that, for every $b\in V$,  $a_{(n)}b=0$ eventually for $n\to +\infty$. Moreover, for every $a \in V$ we have $a_{(n)}\Omega = 0$ for all $n \geq 0 $  and    $a_{(-1)}\Omega = a$.

The vertex operator $Y(\nu,z)$ corresponding to the conformal vector $\nu$ is called the {\it energy-momentum field} of $V$ and is often written as 
\begin{equation} 
Y(\nu,z)= \sum_{n \in \ZZ} L_n z^{-n-2} 
\end{equation}
so that $L_n=\nu_{(n +1)}$. The endomorphisms $L_n$, $n\in \ZZ$, give a representation of the Virasoro algebra on $V$ with  
central charge $c\in \CC$ (the central charge of $V$) i.e.\

\begin{equation}
\label{Virasoro_Equation}
{[} L_n,L_m{]} = (n-m) L_{n+m} +\frac{c}{12}(n^3-n)\delta_{n,-m} 1_V
\end{equation}
for all $n,m\in\ZZ$. Moreover, $[L_{-1},a_{(n)}] = -n a_{(n-1)}$ for all $a \in V$ and all $n\in \ZZ$. 
Equivalently, 
\begin{equation}
[L_{-1},Y(a,z)] = \frac{d}{dz} Y(a,z)
\end{equation}
for all $a \in V$.  

The operator $L_0$ is diagonalizable on $V$ with integer eigenvalues. Accordingly, if we set $V_n := \mathrm{Ker}(L_0 - n1_V)$, $n \in \ZZ$, we have 
\begin{equation}
V = \bigoplus_{n \in \ZZ} V_n 
\end{equation}  
so that $V$ has a natural $\ZZ$-grading.

In various occasions we will make direct use of the Borcherds identity  \cite{Bor86},
 i.e.\! the equality
\begin{eqnarray}
\label{B-id}
\nonumber
\sum_{j=0}^{+\infty}
\binom{m}{j}
\left(a_{(n+j)}b\right)_{(m+k-j)}c =
\sum_{j=0}^{+\infty}(-1)^j
\binom{n}{j}
a_{(m+n-j)}b_{(k+j)}c \\ -
\sum_{j=0}^{+\infty}(-1)^{j+n}
\binom{n}{j}
b_{(n+k-j)}a_{(m+j)}c \, ,\;\;\;\;
\;a,b,c\in V, \,m,n,k\in \ZZ \, .
\end{eqnarray}
This can be taken as one of the axioms of vertex operator algbras or as a consequence of other axioms including, e.g., locality, see \cite[Sect.\! 4.8]{Kac98}. 

Another useful property of vertex operator algebras is {\it skewsymmetry} 
\begin{equation}
\label{skewsymmetry_equation}
a_{(n)}b = - \sum_{j=0}^{+\infty} \frac{(-1)^{j+n}}{j!}(L_{-1})^j b_{(n+j)}a  
\end{equation}
for all $a,b \in V$ and all $n \in \ZZ$, see \cite[Section 4.2]{Kac98}.

The endomorphisms $a_n \in \mathrm{End}(V)$, $a\in V$, $n\in \ZZ$, are defined by 
\begin{equation}
Y(z^{L_0}a , z) = \sum_{n \in \ZZ} a_{n}z^{-n}.
\end{equation}

If $L_0 a = da$  (equivalently $a \in V_d$), we say that $V$ is homogeneous of conformal dimension (or conformal weight) $d$.  In this case 
we have $a_n = a_{(n + d -1)}$ and $a_{(n)} = a_{n+1-d}$ for all $n\in \ZZ$. In particular, $a_{-d} = a_{(-1)}$ and consequently $a = a_{-d}\Omega$.
The Borcherds identity gives

\begin{equation}
\label{[L_0,a_n]}
[L_0, a_n]=-na_n 
\end{equation}
\begin{eqnarray}
\label{l_-1commutation}
& [L_{-1}, a_n] = (-n -d +1) a_{n-1}  \\
\label{l_1commutation}
& [L_{1},a_n]  =  -(n -d +1) a_{n+1} + (L_1a)_{n+1}  
\end{eqnarray}
for all homogeneous $a\in V_d$ and all $n\in \ZZ$.  

An important consequence of Equation (\ref{[L_0,a_n]}) is that the action of the endomorphisms $a_m$ on $V$ is well-behaved with respect to the $\ZZ$-grading,  namely
\begin{equation}
a_m V_n \subset V_{n-m}
\end{equation}
for all $a \in V$ and all $m,n \in \ZZ$. In particular $a_{-n}\Omega \in V_n$ for all $n \in \ZZ$.

A homogeneous vector $a \in V$ and the corresponding field $Y(a,z)$ 
are called {\it quasi-primary} if $L_1a=0$ and {\it primary} 
if $L_na=0$ for every integer $n>0$. Since 
$L_n\Omega = \nu_{(n+1)}\Omega =0$ for every integer $n\geq -1$, 
the vacuum vector $\Omega$ is a primary vector in $V_0$.
Moreover, it follows by the Virasoro algebra relations that 
the conformal vector $\nu$ is a quasi-primary vector in $V_2$.

We have the following commutation relations:
\begin{equation}
\label{EqQuasi-Primary/PrimaryCommutation}
[L_{m}, a_n] = \left( (d -1)m- n \right) a_{m+n},  
\end{equation}
for all primary (resp. quasi-primary) $a \in V_d$, for all $n\in \ZZ$ and all $m \in \ZZ$ (resp. $m\in \{-1,0,1\}$), see, e.g., \cite[Cor.4.10]{Kac98}.

If $a, b \in V$ and $a$ is homogeneous of conformal dimension $d_a$, then it follows directly from the Borcherds  identity in Equation (\ref{B-id}) that

\begin{eqnarray}
\label{B-id_hom}
\nonumber
\sum_{j=0}^{+\infty} \binom{m}{j}(a_{n+j}b)_{m+k}  = 
\sum_{j=0}^{+\infty}(-1)^j
\binom{n +d_a -1}{j}  
a_{m+n-j}b_{k -n +j}  \\
+ \sum_{j=0}^{+\infty}(-1)^{j + n+d_a}
\binom{n +d_a -1}{j}   b_{k - j + d_a -1}a_{m+ j + 1 - d_a} \, ,
\;\;\;\;
\,m,n,k\in \ZZ \, .
\end{eqnarray}
By taking $n=1-d_a$, $k= q +1-d_a$  and $m=p+d_a -1$   in Equation (\ref{B-id_hom}) we obtain the Borcherds commutator formula

\begin{equation}
\label{commutator_formula_equation} 
[a_{p}, b_{q}] = \sum_{j=0}^{+\infty} \binom{p+d_a -1}{j} (a_{j+1-d_a}b)_{p+q}\, , \; p,q \in \ZZ \, .
\end{equation}
Moreover, by taking $m=0$ in Equation (\ref{B-id_hom}) we get 
\begin{equation}
\label{B-id_hom2}
(a_nb)_k=
\sum_{j=0}^{+\infty}(-1)^j
\binom{n +d_a -1}{j}\left( 
a_{n-j }b_{k -n +j} + (-1)^{n +d_a }b_{k-j + d_a -1}a_{j + 1 - d_a}\right),
\end{equation}
for all $n,k  \in \ZZ$.
\medskip

A homogeneous vector $a \in V$ is called a {\it Virasoro vector} (with central charge $\tilde{c} \in \CC$) if  
\begin{equation}
Y(a,z) = \sum_{n \in \ZZ} \tilde{L}_n z^{-n -2}
\end{equation} 
and the endomorphisms  $\tilde{L}_n$, $n \in \ZZ$, give a representation of the Virasoro algebra on $V$ with central charge $\tilde{c}$. 
Then $a_{(n+1)}=\tilde{L}_n$ for all $n\in \ZZ$ so that $a =  \tilde{L}_{-2}\Omega  \in V_2 $ and $\tilde{L}_n\Omega = 0$ for all $n \geq -1$. The conformal vector is a quasi-primary Virasoro vector and, if  
$V_1 = \{0\}$, then every Virasoro vector is quasi-primary. However, a Virasoro vector needs not to be in general quasi-primary. 
If $a$ is a Virasoro vector, we say that $Y(a,z)$ is a {\it Virasoro field}. 
\medskip

Let $X \subset V$ be a (vector) subspace. We say that $X$ is a  {\it $\ZZ$-graded subspace} of $V$ if 
$X = \bigoplus_{n \in \ZZ}X_n$, where $X_n := X \cap V_n$, $n \in \ZZ.$ Equivalently, $X$ is $\ZZ$-graded if 
$L_0 X \subset X$.
\smallskip 

 A {\it vertex subalgebra} of $V$ is a subspace $U \subset V$ such that $\Omega \in U$ and $a_{(n)}b \in U$ for all $a, b \in U$ and all 
$n \in \ZZ$. A vertex subalgebra of $V$ is always a vertex algebra but in general not a vertex operator algebra. 
If $U$ is a $\ZZ$-graded subalgebra of  $V$ and it has a structure of vertex operator algebra which is compatible with the 
 $\ZZ$-grading inherited from $V$, then we say that $V$ is a {\it vertex operator algebra extension} of $U$. 
If $\nu \in U$ then  $U$ is  a  $\ZZ$-graded subalgebra of $V$ and it is also a vertex operator algebra. In this case we  say that $V$ is a vertex operator algebra extension of $U$ {\it with the same conformal vector}. 

Given a subset $\mathscr{F} \subset V$,  we set $U_\mathscr{F}$ to be the intersection 
of all vertex subalgebras of $V$ containing $\mathscr{F}$. Then $U_{\mathscr{F}}$ is the smallest vertex subalgebra of $V$ containing ${\mathscr{F}}$ and we call it the {\it vertex subalgebra generated by ${\mathscr{F}}$}. If $U_{\mathscr{F}} = V$ we say that {\it $V$ is generated by $\mathscr{F}$}.  If all vectors in $\mathscr{F}$ are homogeneous, then $U_{\mathscr{F}}$ is a $\ZZ$-graded vertex subalgebra of $V$.

\medskip

We now come to unitary vertex operator algebras, cf.\ \cite[Chapter 5]{CKLW18}. In order to simplify the discussion we will assume that $V$ is of {\it CFT-type} i.e.\ that $V_0 = \CC \Omega$ and that $V_n =\{0 \}$ for $n < 0$. Actually, the second condition follows from the first, see \cite[Remark 4.5]{CKLW18}.  Note that by 
\cite[Proposition 5.3]{CKLW18} a unitary vertex operator algebra is of CFT-type if and only if it is simple.

Let $(\cdot | \cdot ) : V \times V \to \CC$ be a scalar product on $V$ (i.e.\ a positive definite Hermitian form on $V$). We say that $(\cdot | \cdot )$ is {\it normalized}  if 
$(\Omega| \Omega ) =1$.  Moreover, we say that $(\cdot | \cdot )$ is {\it invariant} if there exists an anti-linear involution 
$V \ni a \mapsto a^* \in V$ such that $\nu^* = \nu$ and
\begin{equation} \label{invariant_scalar_product}
(b|a_n c) =  (a^*_{-n}b |c)
\end{equation}
for all $a,b,c \in V$ and all $n \in \ZZ$.  As a consequence, we also have that $\Omega^* = \Omega$.  

If $V$ is equipped with an invariant and normalized scalar product $(\cdot | \cdot)$, we say that $V$ is a {\it unitary vertex operator algebra}.  If this is the case then, by \cite[Theorem 5.16]{CKLW18}, there exists a necessarily unique vertex operator algebra anti-linear automorphism 
$\theta: V \to V$ (the {\it PCT-operator}) such that 
\begin{equation}
a^* := e^{L_1}(-1)^{L_0}\theta(a) 
\end{equation} 
for all $a \in V$. It turns out that $\theta ^2 = 1_V$, cf. \cite[Proposition 5.1]{CKLW18}. If $V$ is unitary, then the representation of the Virasoro algebra associated with the conformal vector is unitary, \cite{CKLW18,KaRa}. In particular, the central charge $c$ of $V$ is a nonnegative real number.

The unitary structure is not necessarily unique but is unique up to 
unitary vertex operator algebra automorphisms \cite[Proposition 5.19]{CKLW18}. More precisely, if $\{\cdot | \cdot \}$ and $(\cdot | \cdot )$ are two 
normalized invariant scalar products on  $V$, then there is a vertex operator algebra automorphism $g: V \to V$ such that 
$\{\cdot | \cdot \} = (g\cdot | g \,\cdot )$. Conversely, if $(\cdot | \cdot )$  is a normalized invariant scalar product on $V$ 
and $g: V \to V$ is a vertex operator algebra automorphism, then $(g \cdot | g \cdot )$ is a normalized invariant scalar product with anti-linear involution given by $a \mapsto a^{*_g} := g^{-1}(ga)^*$, $a \in V$ and PCT-operator $\theta_g:=g^{-1}\theta g$. 
\medskip

We now discuss unitary subalgebras, cf.\ \cite[Section 5.4]{CKLW18} and \cite{CGH19}. Let $V$ be a simple unitary vertex operator algebra. We say that $a \in V$ is a
{\it Hermitian vector} if $a=a^*$. If $a \in V$ is Hermitian we also say that $Y(a,z)$ is a {\it Hermitian field}. If $a$ is a quasi-primary vector of conformal dimension $d$, then $a^* = (-1)^d\theta a$ so that $a$ is Hermitian if and only if $a= (-1)^d\theta a$.  More generally, we say that a subset 
$\mathscr{F} \subset V$ is {\it Hermitian} if 
$$\mathscr{F}^* = \{a \in V: a^* \in \mathscr{F} \} \subset \mathscr{F}$$ 
(equivalently $\mathscr{F}^* = \mathscr{F}$ ). 
A {\it unitary vertex subalgebra} (or simply a {\it unitary subalgebra}) $U$ of $V$ is a Hermitian $\ZZ$-graded vertex subalgebra of $V$.  
By \cite[Proposition 5.23]{CKLW18} a vertex subalgebra $U \subset V$ is unitary if and only if $L_1U \subset U$ and $\theta U \subset U$. If 
$\mathscr{F} \subset V$ is a  family of Hermitian quasi-primary vectors, then $U_{\mathscr{F}}$ is a unitary subalgebra of $V$. 
\medskip 

If $U \subset V$ is a unitary subalgebra then the coset subalgebra 
\begin{equation}
U^c:= \{b \in V: a \in U \Rightarrow [Y(b,z),Y(a,w)] = 0 \} 
\end{equation}
is also a unitary subalgebra, see \cite[Example 5.27]{CKLW18}. Moreover, there exists a unique $\ZZ$-grading preserving vector space endomorphism $e_U:  V \to V$ such that $e_U^2 =e_U$, $e_UV=U$ and $(e_Ub|a)=(b|e_Ua)$ for all $a,b \in V$ (the orthogonal projection of $V$ onto $U$). Then,  by 
\cite[Proposition 5.29]{CKLW18}, $\nu^U := e_U \nu$ is a Hermitian quasi-primary Virasoro vector in $U$ with nonnegative central charge $c_U \leq c$.  In particular,  $U$ is a simple unitary vertex operator algebra with $\ZZ$-grading inherited from $V$ and we say that $V$ is a {\it unitary vertex operator algebra extension} of $U$.  Moreover, $\nu = \nu^U + \nu^{U^c}$ and, if we set $L^U_n := \nu^U_n$  and $L^{U^c}_n := \nu^{U^c}_n$, $n \in \ZZ$, then $L_0, L^{U}_0$ and  $L^{U^c}_0$ are simultaneously diagonalizable in $V$ with nonnegative eigenvalues,  
$L_0 = L^{U}_0 +L^{U^c}_0$ and $c=c_U + c_{U^c}$. It follows that $(b|(L^U_0 +1_V)^s b) \leq (b|( L_0 +1_V)^s b)$ for all $b\in V$ and all real numbers $s\geq 0$.  Furthermore, $\nu^U= \nu$ if and only if $c=c_U$. In this case, we say that  $V$ is a  unitary vertex operator algebra extension of $U$ {\it with the same conformal vector}, cf.\ \cite{CGGH,Gui21b}.

\end{section}

\begin{section}{Energy bounds} 
\label{section_bounds}

In this section we discuss the notion of energy bounds for unitary vertex operator algebras, cf.\ \cite[Chapter 6]{CKLW18}. 

Let $V$ be a simple unitary vertex operator algebra. Then $V$ is a normed space with norm defined by $\|a\| := \sqrt{(a|a)}$.
 Let $X \subset V$ be a subspace of $V$, let $\mathcal{X}$ be a normed space and let $R: X \to \mathcal{X}$ be a linear map. 
We define the extended real valued norm $\|R\| \in [0,+\infty]$ of $R$ by 
\begin{equation}
\| R \| := \sup_{\{a\in X: \, \|a\| \leq 1\}} \|Ra \| .
\end{equation} 
Clearly, $\|R\|<+\infty$ whenever $X$ is finite dimensional. Note also that $\|R a \| \leq \|R\|  \|a\|$ for all $a\in V$ and that if $X^1, X^2$ are subspaces of 
$V$ and $R_2: X^2 \to X^1$ and  $R_1 : X^1 \to \mathcal{X}$ then $\| R_1 R_2\| \leq \|R_1\| \|R_2\|$.    

There is a natural filtration on $V$ given by the increasing sequence of  subspaces $V_{\leq n} \subset V$, $n \in \ZZ$, defined by
\begin{equation}
V_{\leq n} = \bigoplus_{k \leq n}V_k.  
\end{equation} 
Then $V_{\leq n} = \{ 0\}$ whenever $n <0$ (because $V$ is of CFT-type) and 
$$
V = \bigcup_{n \in \ZZ}V_{\leq n}\,. 
$$
Moreover, $a_m V_{\leq n} \subset V_{\leq n-m}$ for all $a \in V$ and all 
$m,n \in \ZZ$.
\smallskip

For $R \in \mathrm{End}(V)$ and $n \in \ZZ$, we set 
\begin{equation}
\|R \|_n := \| R \restriction_{V_{\leq n}} \| 
\end{equation}
so that $\|R \|_n =0$ for $n<0$, $n\mapsto \| R\|_n$ is an increasing sequence and   
\begin{equation}
\|R \| = \lim_{n \to +\infty} \|R\|_n =   \sup_{n \in \Zgeq} \|R\|_n =  \sup_{n \in \ZZ} \|R\|_n \,.
\end{equation}
If there is a $R^\dagger \in \mathrm{End}(V)$ such that $(a|Rb) = (R^\dagger a | b)$ for all $a,b \in V$, then we have the
inequality $\|R \|_n^2  \leq  \|R^\dagger R\|_n$ for all $n \in \ZZ$.  Moreover, if for a given $n\in \ZZ$ it holds 
$R^\dagger R V_{\leq n} \subset V_{\leq n}$, then we also have the {\it $C^*$-identity}  
$\|R \|_n^2  =  \|R^\dagger R\|_n$.  In particular, since $(a_m)^\dagger = a^*_{-m}$ we have 
\begin{equation}
\label{equation_a*estimate1}
\| a^*_{-m}a_m\|_n = \|a_m \|_n^2
\end{equation}
for all $a\in V$ and all $m, n \in \ZZ$. As a consequence, 
\begin{equation}
\label{equation_a*estimate2}
\|a_{m} \|_n = \|a^*_{-m}\|_{n-m}
\end{equation}
for all $a\in V$ and all $m,n \in \ZZ$.

\begin{defin} (\cite[Chapter 6]{CKLW18} and \cite[Section 4.4]{Gui21a}) Let $V$ be a simple unitary vertex operator algebra and let $s$ be a nonnegative real number. We say that a vector $a \in V$ {\it satisfies $s$-th order (polynomial)  energy bounds} if there exist real numbers $t\geq 0$, and $C \geq 0$ such that $\|a_m\|_n \leq C (1+ |m|)^t (1 + |n|)^s$ for all $m,n \in \ZZ$ (equivalently $\| a_m (L_0+1_V)^{-s}\| \leq C (1+ |m|)^t $  for all $m \in \ZZ$). If $a$ satisfies $1$-st order energy bounds we also say that $a$ {\it satisfies linear energy bounds}. We say that $a \in V$ {\it satisfies energy bounds} if it satisfies $s$-th order energy bounds for some real number $s\geq 0$.      
\end{defin}

\begin{rem} 
\label{smeared_remark}
Let $\mathcal{H}_V$ be the Hilbert space completion of $V$ and let $a \in V$ satisfy $s$-th order energy bounds. For a smooth complex valued function 
$f \in C^\infty(S^1)$ with Fourier coefficients  $\hat{f}_n := \frac{1}{2\pi} \int_{-\pi}^{\pi} f(e^{i \vartheta}) e^{-in\vartheta} d\vartheta$ the series 
$\sum_{n \in \ZZ} \hat{f}_n a_n b$ converges in $\mathcal{H}_V$ for all $b \in  V$. Accordingly, one can define a linear map $Y_0(a,f): V \to \mathcal{H}_V$ (the {\it smeared vertex operator}). Moreover,  given a real number $s \geq 0$, then the real numbers $t \geq 0$, and $C \geq 0$ are such that $\|a_m\|_n \leq C (1+ |m|)^t(1 + |n|)^s$  for all $a\in V$ and all $m, n \in \ZZ$ 
if and only if $\|Y_0(a,f) (L_0 + 1_V)^{-s} \| \leq C\|f\|_t $, where $\| f \|_t = \sum_{m \in \ZZ}(1 + |m|)^t |\hat{f}_m|$.  The smeared vertex operators play a fundamental role in the connection between unitary vertex operator algebras and conformal nets, see \cite{CKLW18}. 
\end{rem}

\begin{rem}
\label{bounded_remark}
Let $V $ be a simple unitary vertex operator algebra. Then it follows from \cite[Theorem 1]{Bau97} that $a \in V$ satisfies zeroth order energy bounds if and only if $a \in \CC\Omega$.
\end{rem}

Well known examples of  elements of $V$ satisfying energy bounds are the Hermitian quasi-primary Virasoro vectors and vectors in $V_1$, c.f.\ 
\cite[Proposition 6.3]{CKLW18}. Here we give some more details. Let us start with the Virasoro vectors. Let  $\tilde{\nu}\in V$ be a quasi-primary Hermitian Virasoro vector. Moreover, let
\begin{equation}
Y(\tilde{\nu},z) =\sum_{n \in \ZZ} \tilde{L}_n z^{-n-2}
\end{equation} 
be the corresponding Virasoro field with central charge $\tilde{c}$. Let $U_{\{\tilde{\nu}\}}$ be the vertex subalgebra generated by  $\tilde{\nu}$. It is a unitary subalgebra of $V$ with  $U_{\{\tilde{\nu}\}} \cap V_2 =\CC\tilde{\nu}$. It follows that $\tilde{\nu} = \nu^U$, so that  $\tilde{L}_0$ is diagonalizable on $V$ with 
nonnegative eigenvalues, $\|(\tilde{L}_0 +1_V)b\| \leq \|({L}_0 +1_V)b\|$ for all $b \in  V$ and $0 \leq \tilde{c} \leq c$.
Then by \cite[Lemma 4.1]{CW05} there is a positive real number $r(\tilde{c}) \geq 1$, depending on the central charge 
$c$ of $V$ only, such that 
\begin{equation}
\|\tilde{L}_m b\| \leq r(\tilde{c})(1+ |m|)^\frac32 \|(\tilde{L}_0 +1_V)b\| \leq r(\tilde{c})(1+ |m|)^\frac32 \|(L_0 +1_V)b\|    
\end{equation}
for all $b \in V$.  Actually,  it follows from \cite[Equation 2.8]{GW85} that one can take 
$r(\tilde{c}) = 1 + \sqrt{\frac{\tilde{c}}{3}}$ (note that the choice in \cite[Proposition 3.2]{CKLW18} is not correct although the mistake has no consequences for the results in \cite{CKLW18}).
We record this in the following proposition.

\begin{prop}
\label{boundVirProp}
Let $V$ be simple unitary  vertex operator algebra with central charge $c$ and let $a \in V$ be a quasi-primary Hermitian Virasoro vector with central charge $\tilde{c}$. Then  $a$ satisfies linear energy bounds. More precisely
$$\|a_m\|_n \leq (1 + \sqrt{\tilde{c}/3})(1+|m| )^{\frac32}(1 + |n|)$$
for all $m, n \in \ZZ$.
\end{prop}

Before discussing energy bounds for  vectors in $V_1$, we state and prove a lemma which will play a central role in Section \ref{section_main}.

\begin{lem}
\label{productLemma}
Let $V$ be a simple unitary vertex operator algebra.
If $a \in V_d$  then $\|a_mb\|^2 \leq (b|(a_{-d}a^*)_0 b)$ for all $b \in V$ and all $m\in \Zgeq$. Hence,   
$\|a_m\|_n^2 \leq \| (a_{-d}a^*)_0 \|_n$ for all $m \in \Zgeq$ and all $n \in \ZZ$. 
In particular, $\|a_0\|_n^2 \leq \| (a_{-d}a^*)_0 \|_n$ for all $n \in \ZZ$\,. 
\end{lem}

\begin{proof}

By Equation (\ref{B-id_hom2}) we have 

\begin{equation}
(a_{-d}a^*)_0 = \sum_{j=d}^{+\infty}a_{-j}a^*_j + \sum_{j = 1-d}^{+\infty} a^*_{-j}a_j
\end{equation}
and hence
\begin{equation}
(b|(a_{-d}a^*)_0b) = \sum_{j=d}^{+\infty}\|a^*_jb\|^2 + \sum_{j = 1-d}^{+\infty} \|a_jb\|^2 
\end{equation}
for all $b \in V$.  If $d >0$ we find  $\|a_mb\|^2 \leq (b|(a_{-d}a^*)_0 b)$ for all $b \in V$ and all $m\in \Zgeq$ and thus 
$\|a_m\|_n^2 \leq \| (a_{-d}a^*)_0 \|_n$ for all $m \in \Zgeq$ and all $n \in \ZZ$. 
If $d=0$ then $a = (\Omega | a) \Omega$, $a^* = (a|\Omega) \Omega$ and $(a_0a^*)_0 = | (\Omega |a)|^2  1_V$ and hence 
$\|a_m\|_n = 0$ for all $m \in \ZZ \setminus \{0\}$ and all $n \in \ZZ$. If $n<0$ we have $\|a_0\|_n = 0$ again. Finally, if $n\geq 0$ then $\|a_0 \|_n^2 =|(\Omega | a)|^2 = \|(a_0a^*)_0\|_n$.   
\end{proof}

We now discuss energy bounds for vectors in $V_1$. It is well known that $V_1$ is a Lie algebra 
with brackets $[a,b] := a_0b$. This follows directly from skewsymmetry (Equation (\ref{skewsymmetry_equation})) and the Borcherds commutator formula (Equation (\ref{commutator_formula_equation})). Moreover, we have the Kac-Moody algebra relations
\begin{equation}
\label{Kac-Moody_equation}
[a_m,b_n] = [a,b]_{m+n}  + m(a^*|b)\delta_{m,-n} 1_V
\end{equation}
for all $a,b \in V_1$ and all $m,n \in \ZZ$. In particular, we have the Heisenberg algebra relations
\begin{equation}
\label{Heisenberg_equation}
[a_m,a_n] = m(a^*|a)\delta_{m,-n} 1_V
\end{equation}
for all $a \in V_1$ and all $m,n \in \ZZ$.

\begin{prop}
\label{boundV1Prop}
Let $V$ be a simple unitary vertex operator algebra and let $a \in V_1$. Then $a$ satisfies 
$\frac12$-th order energy bounds. 
More precisely
$$
\|a_m\|_n \leq  2^{\frac32} \|a\| (1+|m|)^{\frac12}(1+ |n| )^{\frac12}
$$
for all $m,n \in \ZZ$. As a consequence, $a$ satisfies linear energy bounds.
\end{prop}

\begin{proof}
Let $a \in V_1$. If $a=0$ the claim is trivial. For $a\neq 0$ we first assume that $a$ is Hermitian. Note that $a_0a = [a,a]=0$. It is well-known that  
$$
\tilde{\nu} = \frac{1}{2\|a\|^2 } a_{-1}a
$$ 
is a Virasoro vector with central charge $1$, see \cite[Section 5.7]{Kac98} and 
\cite[Section 2.3]{KaRa}.  Moreover, $\tilde{\nu}$ is Hermitian because $a$ is. Finally, since $a$ must be primary, 
$L_1a_{-1}a= a_0a=0$ so that $\tilde{\nu}$ is a Hermitian quasi-primary Virasoro vector. By Lemma \ref{productLemma}
we see that 
$$
\|a_m \|_n^2 \leq 2\|a\|^2\|\tilde{L}_0\|_n  \leq 2\|a\|^2 \|L_0\|_n \leq 2\|a\|^2 |n|  
$$
for all $m \in \Zgeq$ and all $n \in \ZZ$. If $-m\in \Zgeq$ then, by Equation (\ref{equation_a*estimate2}), and recalling that $a=a^*$, we find 
$$
\|a_m\|_n^2 = \|a_{ - m}\|_{n+|m|}^2 \leq 2 \|a\|^2 (|n| + |m| ) \leq 2\|a\|^2 (1+|m|)(1+ |n| ) \,.
$$
Accordingly, if $a \in V_1$ is Hermitian, then 
$$\|a_m\|_n \leq \sqrt{2}\|a\| (1+|m|)^{\frac12}(1+ |n| )^{\frac12}$$
for all $m, n \in \ZZ$. If $a$ is not Hermitian, we can write  
$$
a = \frac12 (a+ a^*) + i \frac{1}{2}(ia^*-ia)
$$ 
and thus 
\begin{eqnarray*}
\|a_m\|_n  &\leq& \frac12\|(a + a^*)_m\|_n  +   \frac12\|(ia^* -i a)_m\|_n\\
&\leq& 2^{\frac32} \|a\| (1+|m|)^{\frac12}(1+ |n| )^{\frac12}
\end{eqnarray*}
for all $m, n \in \ZZ$. 
\end{proof}

We now discuss some results on energy bounds for primary vectors with conformal dimension $d \neq 1$, 
cf. \cite{CTW22}.

\begin{lem} \label{LemmaPrimary1}
Let $V$ be a simple unitary vertex operator algebra  and let $a \in V$ be a primary vector with conformal dimension $d_a \neq 1$.
Then there is a real number $A \geq 0$ such that 
$$\|a_m\|_n \leq  A \sqrt{1 + |m|} (1+|n|) (\|a_0\|_n  + \|a_0\|_{n-m} )$$
for all $m, n \in \ZZ$. Moreover, if  $\|a_0 (L_0 + 1_V)^{-s}\| < +\infty$ for some $s \geq 0$ then $a$ satisfies $(s+1)$-th order energy bounds and 
$\|a_m  (L_0 + 1_V)^{-s-1}\| \leq 2A (1+|m|)^{s + \frac12} \|a_0 (L_0 +1_V)^{-s}\|$ for all $m\in \ZZ$. 
\end{lem}
\begin{proof} If $d_a =0$ then $a = (\Omega | a)\Omega$, so that $a_m = \delta_{m,0} (\Omega |a) 1_V$  
and the statement trivially holds with $A \geq \frac12$. Let us now consider the case $d_a>1$. We follow the argument in the first part of the proof of 
\cite[Proposition 3.1]{CTW22}. 
 If $m=0$ the inequality is trivially satisfied for any $A \geq \frac12$. Let us now assume that $m\neq 0$. 
Since $a$ is primary we have $[L_m,a_0] = (d_a-1)m a_m $ and since $d_a \neq 1$ we have 
\begin{eqnarray*}
\|a_m\|_n &\leq& \frac{1}{|m|(d_a-1)} (\|L_m a_0 \|_n + \|a_0L_m \|_n )   
\end{eqnarray*}
for all $m,n \in \ZZ$, $m\neq 0$.  By Proposition \ref{boundVirProp} 
we have
$$
\|L_m\|_n \leq (1 + \sqrt{c/3})(1+ |m|)^\frac32(1+|n|)     
$$
for all $m, n \in \ZZ$. Hence, 
\begin{eqnarray*}
\|a_m\|_n &\leq& \frac{2(1 + \sqrt{c/3})}{d_a-1} \sqrt{1 + |m|} (1+|n|) (\|a_0\|_n  + \|a_0\|_{n-m} ) 
\end{eqnarray*}
which proves the first claim by taking  $A = \frac{2(1 + \sqrt{c/3})}{d_a-1} + \frac12$. 
Now, we have 
\begin{eqnarray*}
\|a_0\|_n  + \|a_0\|_{n-m} &\leq& \|a_0(L_0 + 1_V)^{-s}\|\left( (1 + |n|)^s  + (1 + |n-m|)^s \right) \\
&\leq &  \|a_0(L_0 + 1_V)^{-s}\|\left( (1 + |n|)^s  + (1 + |n| + |m|)^s \right) \\
&=&  \|a_0(L_0 + 1_V)^{-s}\| (1 + |n|)^s  \left(1 +  \left(\frac{1 + |n| + |m|}{1 + |n|} \right)^s  \right) \\
&\leq& 2 \|a_0(L_0 + 1_V)^{-s}\|(1 + |n|)^s (1+|m|)^s  
\end{eqnarray*}
for all $m \in \ZZ$ and the second claim follows. 
\end{proof} 
Although this will not be needed in the rest of the paper, we state and prove a more sophisticated variant of Lemma  \ref{LemmaPrimary1} which is based on the functional analytic methods developed in \cite{CTW22} and it is of independent interest.

\begin{prop} 
\label{PropositionPrimary}
(cf.\ \cite[Proposition 3.1]{CTW22}) Let $V$ be a simple unitary vertex operator algebra and let $a \in V$ be a primary vector of conformal dimension $d \neq 1$. Assume that $\|a_0(L_0 + 1_V)^{-s}\| < +\infty$ for some $s>0$. Then there exist real numbers $C \geq 0$ and $t \geq 0$ such that $\|a_m(L_0 + 1_V)^{-s}\| \leq C(1 + |m|)^t $ for all $m \in \ZZ$. Hence, $a$ satisfies $s$-th order energy bounds.
\end{prop}

\begin{proof} If $d=0$ then $a_m=0$ for all $m \neq 0$ and the statement is trivial. Let us consider the case $d>1$.  We will make use of the smeared vertex operators in Remark \ref{smeared_remark}. First note that $\|a^*_0(L_0 + 1_V)^{-s}\| =  \|a_0(L_0 + 1_V)^{-s}\| < +\infty$. 
By \cite[Proposition 3.1]{CTW22} the operator $Y_0(a,f)(L_0 +1_V)^{-s}: V \to \mathcal{H}_V$ extends to 
a continuous linear map $\Phi_a(f):  \mathcal{H}_V \to \mathcal{H}_V$ for all $f \in C^\infty(S^1)$. 
Let $B(\mathcal{H}_V)$ denote the Banach algebra of continuous linear maps from $\mathcal{H}_V$ into 
$\mathcal{H}_V$ with the usual operator norm $\| \cdot \|_{B(\mathcal{H}_V)}$ (in fact a $C^*$-algebra). It easy to see that the map $\Phi_a: C^\infty(S^1) \to B(\mathcal{H})_V$ defined by 
$C^\infty(S^1) \ni f \mapsto \Phi_a(f) \in B(\mathcal{H}_V)$ is linear.  Moreover, by Lemma \ref{LemmaPrimary1}, 
there are real numbers $B \geq 0$ and $q \geq 0$ such that 
$\|\Phi_a(f)(L_0 + 1_{\mathcal{H}_V})^{-1}\|_{B(\mathcal{H}_V)} \leq B \|f \|_q$, 
where 
$\|\cdot \|_q$ is the norm on $C^\infty(S^1)$ defined in Remark \ref{smeared_remark}. Now, $C^\infty(S^1)$ with the usual topology of uniform convergence of functions and their derivatives is a Fr\'echet space. In fact, its topology can be induced by the increasing sequence of norms $\| \cdot \|_k$, $k \in \Zgeq$. Let $f_n $ be a sequence of smooth function converging to $f$ in $C^\infty(S^1)$ and assume that $\Phi_a(f_n)$ converges to $T\in B(\mathcal{H}_V)$ in the norm topology of $B(\mathcal{H}_V)$. Then, for any $b,c \in V$, 
\begin{eqnarray*}
|(b|Tc) - (b|\Phi_a(f)c)| &\leq & |(b|Tc) - (b|\Phi_a(f_n)c)| + |(b|\Phi_a(f_n)c) - (b|\Phi_a(f)c)| \\
& \leq &\|b\| \|c\| \| T - \Phi_a(f_n) \|_{B(\mathcal{H}_V)} + B \|b\| \|(L_0 + 1_V)c)\|f_n - f\|_q  \\
&\to 0& \; \; \text{as} \; n\to + \infty 
\end{eqnarray*}
so that $T = \Phi_a(f)$. Thus, $\Phi_a$ is a closed linear map and by the closed graph theorem 
(see, e.g., \cite[Theorem 2.15]{RudinFA}) 
$\Phi_a$ is continuous. Hence, there exist a real number $C \geq 0$ and an integer $t  \geq 0$ such that 
$\|Y_0(a,f)(L_0 + 1_V)^{-s}\| \leq C \|f\|_t$ for all $f \in C^\infty(S^1)$ and the conclusion follows from Remark \ref{smeared_remark}.
 \end{proof}
 
\begin{rem} It is proved in \cite{CTW22} that if a primary vector $a \in V_d$ with $d>1$ satisfies $s$-th order energy bounds with $s < d-1$ then  $a = 0$. Hence a $(d-1)$-th order energy bound is optimal. 
 \end{rem}
 
 An interesting class of primary vectors with conformal dimension $d>1$ and satisfying energy bounds comes from the Zamolodchikov
 $\mathcal{W}_3$-algebra. Let $V$ be a simple unitary vertex operator algebra. We say that a Hermitian primary vector 
 $\tilde{\mu} \in V_3$ is a {\it Hermitian  $\mathcal{W}_3$ vector} with central charge $\tilde{c}\geq 0$ if there is a quasi-primary Hermitian Virasoro vector $\tilde{\nu}$ such that the coefficients of the formal series 
\begin{equation}
 Y(\tilde{\mu},z) =  \sum_{n \in \ZZ} \tilde{W}_n z^{-n-3}\, , \quad  Y(\tilde{\nu},z) = \sum_{n \in \ZZ} \tilde{L}_n z^{-n-2} 
 \end{equation} 
satisfy the $\mathcal{W}_3$ algebra commutation relation
    \begin{eqnarray*}
    [\tilde{W}_m, \tilde{W}_n] &=& \frac{c}{3\cdot 5!}(m^2-4)(m^2-1)m\delta_{m+n,0} \\
                &+&  b^2(m-n)\tilde{\Lambda}_{m+n} + \left[\frac1{20}(m-n)(2m^2-mn+2n^2-8))\right]\tilde{L}_{m+n}
    \end{eqnarray*}
    for all $m,n \in \ZZ$,  
    where $b^2 = \frac{16}{22+5c}$ and
    $\tilde{\Lambda}_n := (\tilde{\nu}_{-2}\tilde{\nu} - \frac{3}{10}\tilde{\nu}_{-1}\tilde{\nu}_{-1}\tilde{\nu} )_n$, $n\in \ZZ$. 
    
    It is not hard to see that if $\tilde{\mu} \in V$ is a Hermitian $\mathcal{W}_3$ vector and $U_{\tilde{\mu}}$ is the unitary subalgebra   
    of $V$ generated by $\tilde{\mu}$, then $\tilde{\nu} = \nu^{U_{\tilde{\mu}}}$. In particular $\tilde{\nu}$ is determined by 
    $\tilde{\mu}$.  The unitary $\mathcal{W}_3$ vertex operator algebras studied in \cite{CTW22b} give many examples of Hermitian
    $\mathcal{W}_3$ vectors.

   The following proposition follows from the results in \cite[Section 4]{CTW22}. 
   
   \begin{prop}
   \label{boundW3_proposition}  
   Let $V$ be a simple unitary vertex operator algebra and let $a \in V_3$ be a Hermitian $\mathcal{W}_3$ vector with central charge $\tilde{c}$.  Then $a$ satisfies third order energy bounds. Moreover, if $\tilde{c} \geq 1$ then $a$ satisfies the optimal  second order energy bounds.
   \end{prop}

\begin{defin}
\label{energy-bounded_definition} 
Let $V$ be a simple unitary vertex operator algebra. Then $V$ is 
{\it energy-bounded} if every $a \in V$ satisfies energy bounds.
\end{defin}

\begin{defin}
\label{subalgebra_energy-bounded_definition} 
Let $V$ be a simple unitary vertex operator algebra and let $U \subset V$ be a unitary subalgebra. Then $U$ is 
{\it energy-bounded} if every $a \in U$ satisfies energy bounds.
\end{defin}

\begin{rem}
\label{subalgebra_energy-bounded_remark} A unitary subalgebra $U$ of the simple unitary vertex operator algebra $V$ is energy-bounded if and only if $V$ is an energy-bounded unitary $U$-module, cf.\ \cite{CWX,Gui19a} and Section \ref{section_main} below. In particular, if $U$ is an energy-bounded unitary subalgebra of $V$ then it is also  an energy-bounded  simple unitary vertex operator algebra.
\end{rem}

The same proof of \cite[Proposition 6.1]{CKLW18} gives the following.

\begin{prop}
\label{bound_generated_prop}
Let $V$ be a simple unitary vertex operator algebra and let $U$ be a unitary subalgebra of $V$. If $U$ is generated by a 
family  of homogeneous vectors satisfying energy bounds, then $U$ is an energy-bounded unitary subalgebra. In particular, if $\mathscr{F} \subset V$ is a family of Hermitian quasi-primary vectors satisfying energy bounds then 
$U_{\mathscr{F}}$ is an energy-bounded  subalgebra of $V$.
\end{prop}

Proposition \ref{bound_generated_prop} together with Proposition \ref{boundVirProp}, Proposition \ref{boundV1Prop} and Proposition \ref{boundW3_proposition} 
 immediately gives the following theorem. 
 
 \begin{theo} 
 \label{energy_bounded_subalgebra_theorem}
 Let $V$ be a simple unitary vertex operator algebra and let $U$ be a unitary subalgebra of $V$. Assume that 
 $U$ is generated by $U_1$, a family of Hermitian quasi-primary Virasoro vector and a family of Hermitian 
 $\mathcal{W}_3$ vectors. Then $U$ is an energy-bounded unitary subalgebra of $V$.
 \end{theo} 

We end this section by introducing the notion of exponentially energy-bounded unitary vertex operator algebra, which will be needed in Section \ref{section_main}.  

\begin{defin} Let $V$ be a simple unitary vertex operator algebra. 
We say that $V$ is {\it exponentially energy-bounded} if, for every $a\in V$ there is a $q>0$ 
such that $\|a_0q^{L_0}\| < +\infty$.  
\end{defin}

If $a \in V$ satisfies energy bounds, then there are real numbers $C \geq 0$, and $s \geq 0$ such that 
$\|a_0 \|_n \leq C(1 + |n|)^s$ for all $n \in \ZZ$ so that   
$\| a_0 q^{L_0} \|_n \leq C(1 +|n|)^s q^|n|$ for all $q > 0$ and all $n \in \ZZ$.   Hence, if $q \in (0,1)$, 
$$
\| a_0 q^{L_0} \| \leq \sup_{n \in \Zgeq} C(1 +n)^s q^n  < +\infty \, .
$$ 
As a consequence, if $V$ is energy-bounded then it is also exponentially energy-bounded.
The following proposition will be important for our main results in Section \ref{section_main}.

\begin{prop}
\label{proposition_exponentially_bounded} 
Let $V$ be a simple unitary vertex operator algebra. If $V$ is  $C_2$-cofinite then $\|a_0q^{L_0}\| < +\infty$ for all $a\in V$ and all 
$q \in (0,1)$. In particular $V$ is exponentially energy-bounded.  
\end{prop}
\begin{proof}  
Let $q \in (0,1)$.\ It is enough to prove that $\|a_0q^{L_0}\| < +\infty$ for all homogeneous vectors $a \in V$. Let $a$ be homogeneous with conformal 
dimension  $d$. Since $V$ is assumed to be $C_2$-cofinite we can use \cite[Theorem 4.4.1]{Zhu96} to conclude that 
$\operatorname{Tr}((a_{-d}a^*)_0 q^{2L_0}) < +\infty$. Hence, by Lemma \ref{productLemma} 
we find
\begin{eqnarray*}
\|a_0 q^{L_0} \|^2 &=& \| (a_0 q^{L_0})^\dagger(a_0 q^{L_0}) \| \\
&=& \| q^{L_0} a^*_0 a_0 q^{L_0}\|  \\
&\leq& \operatorname{Tr}q^{L_0} a^*_0 a_0 q^{L_0} \\
&\leq& \operatorname{Tr}(q^{L_0} (a_{-d}a^*)_0 q^{L_0} ) \\
&\leq& \operatorname{Tr}((a_{-d}a^*)_0 q^{2L_0} ) \\
&<& + \infty  \,.   
\end{eqnarray*}
\end{proof} 

\end{section}

\begin{section}{Main results} \label{section_main}   

In this section we state and prove our main results on energy bounds for vertex operator algebras extensions. A vertex operator algebra automorphism 
$g$ of a simple unitary vertex operator algebra $V$ is said to be a {\it unitary automorphism} if $(ga|gb)=(a|b)$ for all $a,b \in V$. A vertex operator algebra automorphism is unitary if and only if it commutes with the PCT operator $\theta$, see \cite[Remark 5.18]{CKLW18}. Since the vertex operator algebra automorphisms preserve the conformal vector $\nu$, it follows that if $g$ is unitary then $g a^* = (ga)^*$ for all $a \in V$. If $G$ is a compact group of unitary automorphisms of $V$, then the fixed point subalgebra $V^G := \{a \in V: g \in G \Rightarrow ga =a \}$ is a unitary subalgebra of $V$, cf.\ \cite[Example 2.25]{CKLW18}.

\begin{theo}
\label{VGtheorem} 
Let $G$ be a compact group of unitary automorphisms of the simple unitary vertex operator algebra $V$ and let $U:=V^G$. Assume that $U$ is an energy-bounded 
unitary subalgebra of $V$. Then $V$ is energy-bounded.
\end{theo}

\begin{proof}
Since $V$ is generated by the conformal vector $\nu$ and by primary vectors (see, e.g. , \cite[Remark 3.9]{CTW22}), then, by Lemma \ref{LemmaPrimary1}, it is enough to prove that for every primary vector $a \in V$ there is an $s \geq 0$ such that $\|a_0 (L_0+1_V)^{-s} \| < +\infty$, cf.\  \cite{CTW22}.
Now, if $a \in V_{d}$ is primary, there is an orthonormal basis $\{a^1,a^2, \dots, a^{k_d}  \}$ of $V_d$ with $a= \|a\| a^1$. 
Then, $x:= \sum_{i=1}^{k_d}a^i_{-d}{a^i}^* \in U$ because $g\restriction_{V_d}$ is represented by a unitary matrix in any orthonormal basis.  Since $U$ is an energy-bounded unitary subalgebra of $V$, then $x$ satisfies energy bounds and hence there is a real number  $s \geq 0$ such that 
$\|x_0 (L_0+1_V)^{-2s} \| < +\infty$. By Lemma \ref{productLemma} and the fact that $x_0$ commutes with $L_0$ we find 
\begin{eqnarray*} 
\| a_0 (L_0 + 1_V)^{-s} b \|^2 &=& \|a \|^2    \| a^1_0  (L_0 + 1_V)^{-s}b \|^2  \\ 
&\leq&    \|a\|^2 \sum_{i=1}^{k_d}   \| a^i_0(L_0 + 1_V)^{-s}b \|^2        \\
&\leq& \|a\|^2 \sum_{i=1}^{k_d}   ((L_0 + 1_V)^{-s}b| ( a^i_{-d}{a^i}^*)_0 (L_0 + 1_V)^{-s}b) \\
&=&    \|a\|^2 ((L_0 + 1_V)^{-s}b|x_0 (L_0 + 1_V)^{-s} b) \\
&=&    \|a\|^2 (b|x_0 (L_0 + 1_V)^{-2s} b) \\
&\leq &  \|a\|^2 \|x_0 (L_0+1_V)^{-2s} \|  \|b\|^2 ,
\end{eqnarray*}
for all $b \in V$. Hence, $\|a_0(L_0+1_V)^{-s}\|^2 \leq \|a\|^2 \|x_0 (L_0+1_V)^{-2s} \| < +\infty$ and the conclusion follows.
\end{proof}

As an immediate application, we give a new proof of energy-boundedness of lattice models, see 
\cite[Section 5.3]{Gui21a}.

\begin{prop} Let $L$ be an even positive-definite lattice. Then the lattice vertex operator algebra $V_L$ is energy-bounded. 
\end{prop}
\begin{proof}
If $n$ is the rank of $L$ then the rank $n$ Heisenberg algebra is a unitary subalgebra of $V_L$ which is energy-bounded by 
Proposition \ref{boundV1Prop}.
The automorphism group of $V_L$ contains a compact subgroup $G$ of unitary automorphisms isomorphic to the $n$-dimensional torus $\mathbb{T}^n$  such that $H=V_L^G$ and the conclusion follows from Theorem \ref{VGtheorem}.
\end{proof} 

We now to discuss our second main result on energy bounds for vertex operator algebras extensions.  We start with two lemmas. 

\begin{lem} \label{LemmaPrimary2} Let $V$ be a simple unitary vertex operator algebra.  Let $b \in V$ satisfy energy bounds and let $a \in V$ be a primary vector with conformal dimension $d_a \neq 1$. 
Then there are real numbers $B \geq 0$ and $t \geq 0$ such that
$$\|a_{-m}b_m\|_n \leq B(1+|m|)^t(1+|n|)^t (\|a_0\|_{n-m} + \|a_0\|_{n} ) $$
for all $m, n \in \ZZ$.
\end{lem}
\begin{proof}
By assumption $b$ satisfies energy bounds and thus there are real numbers $K  \geq 0$ and $q \geq 0$ such that  
$\| b_m \|_n \leq K(1 + |m|)^q(1+|n|)^q$ for all  $m, n \in \ZZ$. It then follows from Lemma \ref{LemmaPrimary1} that there is a real number 
$A\geq 0$ such that
\begin{eqnarray*}
\|a_{-m}b_m\|_n &\leq &   \|a_{-m} \|_{n-m} \|b_m\|_n  \\
&\leq&  A \sqrt{1 + |m|} (1+|n-m|) (\|a_0\|_{n-m}  + \|a_0\|_{n}) K(1 + |m|)^q(1+|n|)^q  \\
&\leq& AK (1 + |m|)^{q +\frac32}(1+|n|)^{q + 1}  (\|a_0\|_{n}  + \|a_0\|_{n-m})
\end{eqnarray*}
for all $m,n \in \ZZ$, and the conclusion follows.
\end{proof}

\begin{lem} \label{LemmaPrimary3} Let $V$ be a simple unitary vertex operator algebra.  Let $b \in V_d$ satisfy energy bounds, let $p \in \Zgeq$ 
 and let $a \in V$ be a primary vector with conformal dimension $d_a \neq 1$.
 Then there are real numbers $C\geq 0$ and $r \geq 0$ such that
$$\|(b_{-p}a)_0\|_n \leq C (1+|n|)^r \|a_0\|_{n+d}  $$
for all $n \in \ZZ$.
\end{lem}
\begin{proof}
It follows from Equation (\ref{B-id_hom2})) that
 
\begin{equation}
(b_{-p}a)_0 = \sum_{j=0}^{+\infty} (-1)^j \binom{d - p -1}{j} \left(b_{-p - j}a_{p +j} + (-1)^{d -p} 
a_{d -1 -j}b_{j+1-d} \right).
\end{equation}    

By assumption $b$ satisfies energy bounds and thus there are real numbers $K \geq 0$ and $q \geq 0$ such that  
$\| b_m \|_n \leq K(1 + |m|)^q(1+|n|)^q$ for all  $m, n \in \ZZ$.

If $c \in V_{\leq n}$ then  $a_{p +j}c =0$ for $j + p >  n $,   so that  $\|b_{-p - j}a_{p +j} \|_n =0 $ for $j + p >  n $. 
Similarly, $b_{j+1-d}c =0$ for $j +1 -d > n$. It follows that 
 $\|a_{d -1 -j}b_{j+1-d}\|_n =0$ for $j +1-d > n$.   

By Lemma \ref{LemmaPrimary2} there exist positive real numbers $B$ and $t$ such that
\begin{eqnarray*}
\|a_{d -1 -j}b_{j+1-d}\|_n &\leq&  B (1+ |j+1-d|)^{t}(1 + |n|)^{t} (\|a_0\|_{n + d -1 -j}  +  \|a_0\|_{n}  )  \\
&\leq& 2B (1+ |j+1-d|)^{t}(1 + |n|)^{t}    \|a_0\|_{n + d} 
\end{eqnarray*}
for all $n \in \ZZ$ and all $j\in \Zgeq$. Recalling that $\|a^i_{d -1 -j}b_{j+1-d}\|_n =0$ for $j+1-d >n$, we see that
$$
\|a_{d -1 -j}b_{j+1-d}\|_n \leq 2B(1 + |n|)^{2t}   \|a_0\|_{n + d} $$
for all $n \in \ZZ$ and all $j\in \Zgeq$. On the other hand, 
$$
\| b_{-p - j}a_{p +j} \|_n \leq  K(|p+j| +1 )^q (1+ |n-p-j|)^q \| a_{p +j}  \|_n 
$$
for all $p,j \in \Zgeq$ and all $n\in \ZZ$.
By Lemma \ref{LemmaPrimary1} there exists a positive real number $A$ such that 
$$
 \|a_{p +j}  \|_n \leq 2A \sqrt{1 + |p + j|} (1+|n|) \|a_0\|_n
 $$
 for all $j \in \Zgeq$ and all $n\in \ZZ$. Hence,  recalling that $\| b_{-p - j}a_{p +j} \|_n =0$ for $p+j >n$, we find that 
 $$
\| b_{-p - j}a_{p +j} \|_n \leq  2KA (|n| +1 )^{2q + \frac32}\| a_{0}  \|_n 
$$
for all $j \in \Zgeq$ and all $n\in \ZZ$.  It follows that 
\begin{eqnarray*}
\|(b_{-p}a)_0\|_n &\leq& \left( 2KA (|n| +1 )^{2q + \frac32}\| a_{0}  \|_n  + 2B(1 + |n|)^{2t}   \|a_0\|_{n + d} \right)    
\sum_{j=0}^{n+d} \left| \binom{d - p -1}{j} \right| \\
&\leq& (2KA + 2B)(|n| + 1)^{2q + 2t + \frac32}   \|a_0\|_{n + d} \sum_{j=0}^{n+d} (|d-p -1| +j)^{|d-p -1|}   \\
&\leq& (2KA + 2B)(|n| + 1)^{2q + 2t + \frac32} (|n|+ p +2d +1)^{d+p+2}  \|a_0\|_{n + d} 
\end{eqnarray*}
for all $n\in \ZZ$ and the conclusion follows.
\end{proof}

\begin{theo} 
\label{main_theorem}
Let $V$ be simple unitary exponentially energy-bounded vertex operator algebra and let $U \subset V$ be an energy-bounded  unitary subalgebra of $V$ such that $V$ is a finite direct sum of simple $U$-modules.  
Then $V$ is energy-bounded. 
\end{theo} 

\begin{proof} Let $\tilde{U}$ be the unitary subalgebra of $V$ generated by $U$ and $V_1$. By Proposition \ref{boundV1Prop} every vector in $V_1$ satisfies energy bounds and $U$ is an energy-bounded unitary subalgebra of $V$ by assumption. Therefore, Proposition \ref{bound_generated_prop} implies that $\tilde{U}$ is an energy-bounded unitary subalgebra too.  Since $U \subset \tilde{U}$ then $V$ is a finite direct sum of  irreducible $\tilde{U}$-modules. Hence, we can assume that $V_1 \subset U$ by replacing if necessary $U$ with $\tilde{U}$. We write $V= \bigoplus_{k=1}^{N}M^k $
where the $M^k$, $k=1,\dots,N$, are pairwise orthogonal simple unitary $U$-modules and $M^1=U$. We denote by $M^k_{(0)}$ the top space of $M^k$, i.e. the lowest energy subspace of $M^k$ and by $d_k$ the corresponding lowest energy. Since we are assuming 
that $V_1 \subset U$ we have that $d_k \neq 1$ for all $k \in \{1,\dots,N\}$.

For each $k=1,\dots,N$, pick a vector $a^k \in M^k_{(0)}$ with 
$\|a^k\| = 1$. Accordingly, each $a^k$ is primary of dimension $d_k$ and $V$ is generated by $U \cup \{a^1, \dots , a^N\}$. 
It is thus enough to prove that for every $a^k$ there is a real number  $s>0$ such that $\|a^k_0 (L_0 + 1_V)^{-s}\| < +\infty$.
For every integer $n  \in \Zgeq $ we set 
$$K(n) := \sup_{k} \| a^k_0 \|_n \,. $$
By \cite[Prop. 4.5.6]{LL04}, for any $k \in \{1,\dots,N\}$, there are homogeneous vectors $b^{(k,i,m)} \in  U_{d(k,i,m)}$, $i=1,\dots,N$, 
$m=1,\dots, m(i,k)$  such that

\begin{equation}
a^k_{(-1)}(a^k)^* = \sum_{i=1}^N \sum_{m=1}^{m(i,k)} b^{(k,i,m)}_{d_i-2d_k}a^i
\end{equation}
 so that
 \begin{equation}
 \left(a^k_{(-1)}(a^k)^*\right)_0 = \sum_{i=1}^n \sum_{m=1}^{m(i,k)} \left(b^{(k,i,m)}_{d_i - 2d_k}a^i \right)_0 \,.
 \end{equation}

There is a finite number of vectors $b^{(k,i,m)}$ and these vectors are in $U$ and hence they are energy-bounded. Moreover, 
if $d_i -2d_k >0$ then $b^{(k,i,m)}_{d_i - 2d_k}a^i =0$, because it is a homogeneous vector in $M^i$ with conformal energy less than 
$d_i$. Thus,  it follows from Lemma \ref{LemmaPrimary3} that
 there is an integer $d\geq 0$ and  real numbers
$D > 0$,  $s \geq	0$ such that, for all $k$ and all $n \in \Zgeq$, 
$$\|\left(a^k_{(-1)}(a^k)^* \right)_0\|_n \leq D(n+1)^s K(n+ d)  \, .  $$ 
Thus, it follows by Lemma \ref{productLemma} that
$$K(n)^2 \leq D(n+1)^{s}K(n+d)$$
for all $n \in \Zgeq$. We now set 
$$\alpha_n := \frac{K(n)}{D(1+d)^s (1+n)^s}$$
and we find  $\alpha_n^2 \leq \alpha_{n+d}$. If $\alpha_{\overline{n}} > 1$ for some   $\overline{n}$ then
$\alpha_{\overline{n}+ md} \geq (\alpha_{\overline{n}})^{2^m}$ for all $m \in \Zgeq$ and this is in contradiction with the assumption that $V$ is exponentially energy-bounded. Hence, $\alpha_n \leq 1$ for all $n$ and hence $K(n) \leq C(1+r)^q (1+n)^q$ so that $\|a^k_0 (L_0+ 1_V)^{-q}\| < +\infty$ for all $k$ and 
$V$ is energy-bounded.
\end{proof}

\begin{theo} \label{regularextension_theorem} 
Let $V$ be a simple unitary vertex operator algebra and let $U$ be an energy-bounded unitary subalgebra of $V$ with the same conformal vector. Assume that  $U$ is also a regular vertex operator algebra. Then $V$ is energy-bounded. 
\end{theo}
\begin{proof} Since $U$ is regular and contains the conformal vector $\nu$ of $V$, then $V$ is a finite direct sum of simple $U$-module. Moreover, $V$ is $C_2$-cofinite by \cite[Proposition 5.2]{ABD04} and hence it is exponentially energy-bounded by Proposition \ref{proposition_exponentially_bounded}. Then the conclusion follows from Theorem \ref{main_theorem}
\end{proof}
Thanks to Theorem \ref{energy_bounded_subalgebra_theorem}   we get the following theorem.

 \begin{theo} 
\label{regularextension_theorem2}
 Let $V$ be a simple unitary vertex operator algebra and let $U$ be a unitary subalgebra of $V$ with the same conformal vector. Assume that $U$ is a regular vertex operator algebra and that  $U$ is generated by $U_1$, a family of Hermitian quasi-primary Virasoro vector and a family of Hermitian  $\mathcal{W}_3$ vectors. Then $V$ is energy-bounded.
\end{theo} 
As a representative application we  have the following theorem.

\begin{theo}
\label{theoaffineextension} 
Let $V$ be a simple unitary vertex operator algebra extension with the same conformal vector of a unitary affine vertex operator algebra associated with a semisimple Lie algebra $\mathfrak{g}$.  Then $V$ is unitary and energy-bounded. 
 \end{theo}
 \begin{proof} The unitarity of $V$ follows from \cite[Corollary 4.11]{CGGH}. Now, let $U \subset V$ be the unitary affine subalgebra  of $V$. Since the corresponding Lie algebra is semisimple, then $U$ is rational and $C_2$-cofinite, see 
 \cite[Section 5]{Zhu96} and hence it is regular by \cite[Theorem 4.5]{ABD04}. Moreover, $U$ is generated by $U_1$ and thus the conclusion follows from Theorem \ref{regularextension_theorem2}. 
 \end{proof}
 
 \begin{rem} If $\mathfrak{g}$ is of ADE type, then Theorem  \ref{theoaffineextension}  also follows from \cite[Theorem 2.7.4]{Gui20}.
 \end{rem}
 
 \begin{rem} The energy bounds for the vertex operator algebra extension of  unitary affine vertex operator algebras associated with  semisimple Lie algebras play an important role in the proof of strong locality for the holomorphic vertex operator algebras with central charge $c=24$ and non-zero weight-one subspace, see \cite[Section 5]{CGGH}.
 \end{rem}

 Another rather straightforward consequence of Theorem \ref{regularextension_theorem2}  is the following. 
 
 \begin{theo}
 \label{framed_theorem}
 Let $V$ be a simple unitary framed vertex operator algebra. Then $V$ is unitary and energy-bounded.
 \end{theo}
 \begin{proof} As in the proof of Theorem \ref{theoaffineextension} the conclusion follows from \cite[Corollary 4.11]{CGGH}, 
\cite[Section 5]{Zhu96}, \cite[Theorem 4.5]{ABD04} and    Theorem \ref{regularextension_theorem2}.     
 \end{proof}
 
We conclude this paper by briefly discussing some further applications  of Theorem \ref{main_theorem} which come from the notion of strongly energy-bounded unitary vertex operator algebra  together with recent results by Gui \cite{Gui20}. 

We first recall, for the convenience of the reader, some basic facts about unitary vertex operator algebra modules which have been already partially used in the preceding part of this paper.

Let $V$ be a simple unitary vertex operator algebra with central charge $c$. 
 A {\it vertex algebra module} $M$ for $V$ (or simply a $V$-module) is a vector space 
together with a map  
$$ 
a \mapsto Y^M(a,z) = \sum_{n \in \ZZ} a^M_{(n)} z^{-n-1}   
$$
from $V$ into the family of fields on $M$ such that $Y^M(\Omega,z) = 1_M$ and the following Borcherds identity  for modules holds
\begin{eqnarray}
\label{B-id_module}
\nonumber
\sum_{j=0}^{+\infty}
\binom{m}{j}
\left(a_{(n+j)}b\right)^M_{(m+k-j)}c =
\sum_{j=0}^{+\infty}(-1)^j
\binom{n}{j}
a^M_{(m+n-j)}b^M_{(k+j)}c \\ -
\sum_{j=0}^{+\infty}(-1)^{j+n}
\binom{n}{j}
b^M_{(n+k-j)}a^M_{(m+j)}c \, ,\;\;\;\;
\;a,b\in V,\; c \in M,\;  \,m,n,k\in \ZZ \, .
\end{eqnarray}

A {\it $V$-submodule} $N$ of $M$ is a subspace $N$ of $M$ such that $a^M_{(n)}N \subset N$ for all $a \in V$ and all $n \in \ZZ$. We say that a $V$-module $M$ is {\it irreducible} or {\it simple} if its only $V$-submodules are  $\{0\}$ and $M$.

Similarly to the  case of vertex operator algebras, we write 
\begin{equation}
Y^M(\nu,z) = \sum_{n \in \ZZ} L^M_nz^{-n-2} 
\end{equation} 
and it turns out that the endomorphisms $L^M_n$, $n\in \ZZ$, satisfy the Virasoro algebra commutation relations on $M$ with central charge $c$, 

The endomorphisms $a^M_n \in \mathrm{End}(M)$, $a\in V$, $n\in \ZZ$, are defined by 
\begin{equation}
Y^M(z^{L_0}a , z) = \sum_{n \in \ZZ} a^M_{n}z^{-n}.
\end{equation} 

We say that the $V$-module $M$ is unitary if it is equipped with a scalar product $(\cdot |\cdot)_M$ such that 
\begin{equation} \label{invariant_scalar_product}
(b|a^M_n c) =  ((a^*)^M_{-n}b |c)
\end{equation}
for all $a \in V$, all $b,c \in M$ and all $n \in \ZZ$. In this case all eigenvalues of $L^M_0$ must be nonnegative real numbers. 

Similarly to the case of unitary vertex operator algebras we define the norm $\| \cdot \|_M$ on $M$ by 
$\|c \|_M := \sqrt{(c | c)_M}$, $c \in M$.
If $s \in \Zgeq$ and $a\in V$, we say that the field $Y^M(a,z)$ {\it satisfies $s$-th order energy bounds} if there exist real numbers $C\geq 0$ and $t \geq 0$ such that 
\begin{equation}
\|a^M_n c\|_M \leq  C (1+ |n|)^{t} \|(L^M_0 + 1_M)^s c\|_M 
\end{equation}
for all $c \in M$ and all $n \in \ZZ$. We say that $Y^M(a,z)$ {\it satisfies energy bounds} if it satisfies 
$s$-th order energy bounds for some $s\in \Zgeq$ and we say that $M$ is an {\it energy-bounded module} for $V$ 
if $Y^M(a,z)$ satisfies energy bounds for all $a \in V$.

If $U$ is a unitary subalgebra of a simple unitary vertex operator algebra $V$, then $V$ is in particular a unitary $U$-module. Moreover, $U$ is an energy-bounded subalgebra of $V$ if and only if $V$ is an energy-bounded $U$-module, cf.\  Remark \ref{subalgebra_energy-bounded_remark}.

\begin{defin} (c.f.\ \cite[Section 2.1]{Gui20}) We say that a simple unitary vertex operator algebra $V$ is {\it strongly energy-bounded} if every irreducible unitary $V$-module is energy-bounded. 
\end{defin}

Various interesting examples of strongly energy-bounded simple unitary vertex operator algebras come from regular cosets of affine unitary vertex operator algebras,  see \cite[Section 2.6]{Gui20}. These examples include the discrete 
series $\mathcal{W}$-algebras of ADE type \cite[Section 2.7]{Gui20}. 

\begin{prop}
\label{stronlyEB_regular}
Let $V$ be a simple unitary vertex operator algebra and let $U$ be a unitary subalgebra of $V$. Assume that 
$U$ is regular and  strongly energy-bounded. Then $U$ is an energy-bounded subalgebra of $V$.
\end{prop}
\begin{proof}
The simple unitary vertex operator algebra $V$ is an orthogonal, possibly infinite, direct sum of irreducible $U$-submodules. On the other hand, $U$ being regular, there are only finitely many inequivalent irreducible submodules appearing in the direct sum. It follows that for every $a \in U$ there are real numbers $s \geq 0$ and $C>0$ such that 
$$\|a_0b\| \leq \|(L^U_0 + 1_V)^sb\|  \leq \|(L_0 + 1_V)^s b\|$$ 
for all $b \in V$ and the conclusion follows, e.g. , from Lemma \ref{productLemma}.
  \end{proof}
  
 As a consequence, we have the following generalization of Theorem \ref{regularextension_theorem2}.  
  
 \begin{theo} 
 Let $V$ be a simple unitary vertex operator algebra and let $U$ be a unitary subalgebra of $V$ with the same conformal vector which is also a regular vertex operator algebra. Assume that  $U$ is generated by $U_1$, a family of Hermitian quasi-primary Virasoro vector a family of Hermitian  $\mathcal{W}_3$ vectors and a family of strongly energy-bounded regular unitary subalgebras. Then $V$ is  energy-bounded.
\end{theo}

 \end{section}    

\bigskip
\noindent
{\small
{\bf Acknowledgements.}
We would like to thank Tiziano Gaudio, Bin Gui, Luca Giorgetti, Robin Hillier, Yoh Tanimoto and Mih\'aly Weiner for useful and stimulating discussions. This work was supported in part by the ERC Advanced Grant 669240 QUEST ``Quantum Algebraic Structures and Models''. S.C.\,  is partially supported by  the MIUR Excellence Department Project MatMod@TOV awarded to the Department of Mathematics, University of Rome ``Tor Vergata'', CUP E83C23000330006, by  the University of Rome ``Tor Vergata'' funding \emph{OAQM}, CUP E83C22001800005 and by GNAMPA-INDAM.

\bigskip
\noindent
{\small
\emph{Data sharing is not applicable to this article as no new data were created or analyzed in this study.} }

\end{document}